\documentclass[reqno]{amsart}
\usepackage{tabularx}
\usepackage{halloweenmath}
\usepackage{cancel}
\usepackage{multirow}
\usepackage{fullpage}
\usepackage{array}
\usepackage{arydshln}
\usepackage{hhline}
\usepackage{dashbox}
\usepackage{todonotes}
\usepackage[all,2cell]{xy}
\usepackage[dvipsnames]{xcolor}
\usepackage{graphicx, amsmath, amssymb, amsthm}
\usepackage{newclude}
\usepackage{mathrsfs}
\usepackage{sseq}
\usepackage[hidelinks,backref]{hyperref}
\usepackage{enumitem}
\usepackage{stmaryrd}

\newcommand{\defeq}{\mathrel{:=}}
\let\sec=\S

\DeclareMathOperator{\Tr}{Tr}
\DeclareMathOperator{\tr}{tr}
\DeclareMathOperator{\RU}{RU}
\DeclareMathOperator{\Aut}{Aut}
\newcommand{\m}[1]{\underline{#1}}

\newcommand{\fade}[2][20]{{\color{black!#1}#2}}

\newtheorem{theorem}{Theorem}[section]
\newtheorem*{theorem*}{Theorem}

\newtheorem{lemma}[theorem]{Lemma}
\newtheorem{corollary}[theorem]{Corollary}

\theoremstyle{definition}
\newtheorem{definition}[theorem]{Definition}
\newtheorem{remark}[theorem]{Remark}

\newtheorem{example}[theorem]{Example}

\makeatletter
\providecommand\@dotsep{5}
\renewcommand{\listoftodos}[1][\@todonotes@todolistname]{%
  \@starttoc{tdo}{#1}}
\makeatother

\SelectTips{cm}{}
\UseAllTwocells
\NoCompileMatrices

\DeclareMathOperator{\MGS}{MGS}
\DeclareMathOperator{\Src}{Src}
\DeclareMathOperator{\Sub}{Sub}
\DeclareMathOperator{\dom}{dom}
\newcommand{\M}{\mathcal M}

\begin{document}

\title{Transfer systems give matroids only for cyclic $p$-groups}

\author[Y.~J.~F.~Sulyma]{Yuri~J.~F. Sulyma}
\email{yuri.sulyma@protonmail.com}
\begin{abstract}
  Transfer systems as studied in equivariant algebra admit minimal generating sets, analogous to bases in linear algebra. It is natural to wonder if minimal generating sets form a matroid. We show that this happens only for lattices which are linear orders, or for cyclic groups of prime power order. In this case, we compute some invariants of the resulting matroid.
\end{abstract}

\maketitle

\section{Introduction}

Equivariant algebra and equivariant homotopy theory study generalizations of the \emph{trace} and \emph{norm} maps
\[
  \tr(f) = \sum_{g\in G} g\cdot f,
  \qquad
  N(f) = \prod_{g\in G} g\cdot f
\]
from Galois theory. For example, for a subgroup $H\le G$ there are \emph{transfer} and \emph{norm} maps between the (complex) representation rings $\RU(H)\to\RU(G)$ given by
\[ \tr(V) = V^{\oplus{\Aut(G/H)}}, \qquad N(V) = V^{\otimes{\Aut(G/H)}}, \]
which have similar formal properties, but do not come from a group action. A collection of rings $\m R(G/H)_{H\le G}$ satisfying the same properties as $H\mapsto A^H$ (when $G$ acts on a ring $A$ through ring homomorphisms) or $H\mapsto\RU(H)$ is called a \emph{Tambara functor}, c.f.\ \cite{StricklandTambara} for precise definitions.

While much of commutative algebra carries over, the theory of localization is considerably more subtle. In particular, a fundamental kind of localization is to kill a subgroup (and all transfers generated by it), but this operation does not preserve norms. Concretely, the norm map $N^6_2\colon \RU(C_2) \to \RU(C_6)$ does not pass to a norm map
\[ \tilde N^6_2\colon \RU(C_2)/\tr(\RU(e)) \to \RU(C_6)/\tr(\RU(C_e)); \]
in fact, the only ones that survive are
\[ \tilde N^{p^{m+n}}_{p^n} \colon \RU(C_{p^n})/\tr(\RU(e)) \to \RU(C_{p^{m+n}})/\tr(\RU(e)) \]
for a prime $p$. This led Blumberg and Hill to introduce \emph{incomplete} Tambara functors \cite{BHOperadic,IncompleteTambara}, where we only have norms for a fixed collection of subgroup inclusions $\{K\le H\}$\footnote{However, in \cite{PrismsTambara1} we observed that there \emph{is} an induced map $\tilde N^6_2\colon \RU(C_2)/\tr(\RU(e)) \to \RU(C_6)/\tr(\RU(C_3))$; the existing framework of incomplete Tambara functors does not account for this map.}.

The collection of partial norm structures that are possible produces very rich combinatorics. We now know \cite{NooAssociahedra,BPGenuineOperads,GutierrezOperads,RubinNoo} that they correspond to \emph{transfer systems} on the lattice $\Sub(G)$.

\begin{definition}
  A \emph{transfer system} on $\Sub(G)$ is a sub-poset $T\le\Sub(G)$ such that
  \begin{itemize}
    \item if $(K\le H)\in T$, then $(gKg^{-1}\le gHg^{-1})\in T$ for all $g\in G$;
    \item if $(K\le H)\in T$ and $J\le G$, then $(K\cap J \le H\cap J)\in T$.
  \end{itemize}
  We write $\Tr(G)$ for the poset of transfer systems on $G$. One can similarly define $\Tr(P)$ for any poset $P$ with $G$-action.
\end{definition}

For a poset $P$, we write $P^\to$ for the collection of nontrivial edges of $P$. Given a collection of subgroup inclusions $\{K\subset H\}\subset \Sub(G)^\to$, there is a minimal transfer system containing them. This can be computed explicitly by \emph{Rubin's algorithm} \cite{RubinAlgo}. Conversely, given a transfer system $T$, we can ask for a \emph{minimal generating set} of $T$. Adamyk-Balchin-Barrerro-Scheirer-Wisdom-Zapata Castro \cite{MinimalTransferBases} study this in detail, showing how to compute a minimal generating set for a transfer system by running Rubin's algorithm in reverse. In particular, they show that every minimal generating set of a transfer system has the same cardinality, as is the case for bases of a vector space.

The edges in a minimal generating set are in some sense ``independent''. To unlock more tools for understanding these, it is useful to compare with other notions of ``independence''. \emph{Matroids} are gadgets studied in combinatorics and computer science which abstract the notion of linear independence.

\begin{definition}
    An \emph{independence system} is a pair $(E,I)$ with $E$ a finite set and $I\subset\mathcal P(E)$ a collection of subsets, called \emph{independent} subsets. There are required to satisfy
    \begin{enumerate}
    \item the empty set is independent, i.e.\ $\emptyset\in I$;
    \item every subset of an independent set is independent, i.e.\ $A\subset B\in I\implies B\in I$.
    \end{enumerate}
    An independence system $(E,I)$ is a \emph{matroid} if it further satisfies
    \begin{enumerate}
    \setcounter{enumi}2
    \item if $A,B\in I$ with $|A|<|B|$, then there is some $b\in B\setminus A$ such that $A\cup\{b\}\in I$.
    \end{enumerate}
\end{definition}

\begin{example}
  \label{ex:linear-matroids}
  Let $k$ be a field, let $V$ be a vector space over $k$, and let $E$ be a finite subset of $V$. Let $I$ denote the collection of linearly independent subsets of $E$. Then the pair $(E,I)$ forms a matroid; matroids which arise in this way are called \emph{$k$-linear matroids}.
\end{example}

We would like to try to fit minimal generating sets into this notion of independence.

\begin{definition}
Let $L$ be a finite lattice, let $L^\to$ be the set of nontrivial edges of $L$, and $\MGS(L)\subset\mathcal P(L^\to)$ the collection of minimal generating sets. We say that $L$ is \emph{matroidal} if the pair $(L^\to,\MGS(L))$ is a matroid.
\end{definition}

The pair $(L^\to,\MGS(L))$ is always an independence system, but not always a matroid. In this note we show:

\begin{theorem}
    \label{thm:matroidal}
   A finite lattice is matroidal if and only if it is a linear order, i.e.\ $L\cong[n]$ for some $n$.
\end{theorem}

\begin{remark}
    Given a finite group $G$, we have $\Sub(G)\cong[n]$ if and only if $G\cong C_{p^n}$. Indeed, a finite group with a unique maximal subgroup must be a cyclic $p$-group.
\end{remark}

Transfer systems on $C_{p^n}$ are studied in detail in \cite{NooAssociahedra}, where it is shown that they correspond to points of the associahedron.

While Theorem \ref{thm:matroidal} is a bit disappointing, it is significant in that gives a concrete connection between properties of transfer systems on $L$ and lattice-theoretic properties of $L$. It also gives us two new directions to pursue:

\begin{itemize}
    \item what are the matroid-theoretic properties of the matroid
    \[ \M_n\defeq([n]^\to,\MGS([n]))? \]
    For example, is it $k$-linear for some $k$? There are a whole bunch of matroid invariants we can calculate for this. Can some of these be generalized to arbitrary lattices? This is studied in \sec\ref{sub:lt-matroids}.

    \item there are some generalizations of matroids in the literature. Can any of these accommodate more examples? If not, can we introduce a new generalization which does?
\end{itemize}

\subsection{Acknowledgments}

This research was initiated at the AMS funded MRC on Homotopical Combinatorics, which was partially supported by NSF grant 1916439. We thank the AMS and the organizers of this event for introducing us to this subject. In particular, we thank Katharine Adamyk, Scott Balchin, Miguel Barrero, Steven Scheirer, Noah Wisdom, and Valentina Zapata Castro for conversations related to this work.

\section{Only linear orders give matroids}

The proof of Theorem \ref{thm:matroidal} is fairly short, but requires some lemmas and notation.

\begin{lemma}
  \label{lem:matroidal-implies-linear}
  If $L$ is matroidal, then $L$ is a linear order.
\end{lemma}

\begin{proof}
    Let $x,y\in L$, and consider the diagrams
    \[
    \vcenter{
        \xymatrix{
          x\land y \ar[r] \ar@{->>}[d] & y \ar[d]\\
          x \ar[r] & x\lor y
        }
    }
    \qquad
    \vcenter{
        \xymatrix{
          x\land y \ar[r] \ar[d] \ar@{->>}[dr] & y \ar@{->>}[d]\\
          x \ar[r] & x\lor y
        }
    }
    \]
    If $x$ and $y$ are incomparable, then none of the arrows are identities, so each set of marked arrows is a minimal generating set. However, adding either of the transfers from the second diagram to the first will produce a non-minimal generating set, contradicting axiom (3). Thus, we must have either $x\le y$ or $y\le x$.
\end{proof}

\begin{definition}
    Let $S\subset L^\to$ be a collection of (non-degenerate) edges of $L$. We write
    \[ \Src(S) = \{x \mid (x \to y) \in S\text{ for some }y\in L\} \]
    and call $\Src(S)$ the set of \emph{sources} of $S$. Equivalently, this is the image of $S$ under the domain map $\dom\colon L^\to\subset L^{[1]} \to L$.
\end{definition}

\begin{lemma}
  \label{lem:axiom-3-linear}
  Let $X$ be a transfer system on $[n]$, and let $S$ be a minimal generating set for $X$. Then we have $\Src(X)=\Src(S)$ and $|S|=|\Src(S)|$: that is, $S$ must consist of exactly one edge for each source in $X$.
\end{lemma}

\begin{proof}
  Let $i\le j<k\le n$. The restriction rule gives
  \[ (i\to k) \in S \implies (i\to j)\in S \]
  so no minimal generating set can contain both $i\to j$ and $i\to k$. This shows that $|S|=|\Src(S)|$.

  Next let's examine how the transfer system $X$ gets generated by $S$.
  \begin{itemize}
      \item if $(i\to k)\in S$, then the composition rule gives $(i\to j)\in X$ for all $i\le j\le k$. Note this only adds more edges with source $i$.
      \item if $(i\to j),(j\to k)\in X$, then $(i\to k)\in X$. Again, this only adds edges with source $i$. Note that the edges $(i\to j),(j\to k)$ may not be in $S$, but they are generated by applications of this operation and the previous one, so still cannot add new sources.
  \end{itemize}
  From this we see that $\Src(X)=\Src(S)$.
\end{proof}

\begin{proof}[Proof of Theorem \ref{thm:matroidal}]
  Lemma \ref{lem:matroidal-implies-linear} shows that matroidal lattices are linear orders, so we just need to show that every linear order is matroidal. As noted above, only axiom (3) needs to be checked. So let $A,B\in\MGS(L)$ with $|A|<|B|$. By Lemma \ref{lem:axiom-3-linear}, we have
  \[ |\Src(A)| = |A| < |B| = |\Src(B)|, \]
  so there some $(b\to b')\in B$ with $b\notin\Src(A)$. It follows from the other part of Lemma \ref{lem:axiom-3-linear} that $A\cup\{(b\to b')\}$ is a minimal generating set, which verifies (3).
\end{proof}

\begin{remark}\label{rmk:rank}
  It follows that the rank function of the matroid is given by $r(X) = |\Src(X)|$ for $X\subset[n]^\to$.
\end{remark}

\section{Matroid invariants}
\label{sub:lt-matroids}

%

Let $\M_n=([n]^\to,\MGS([n]))$ be the matroid associated to $[n]$ by Theorem \ref{thm:matroidal}. We call $\M_n$ a \emph{linear transfer matroid}; beware that this is not necessarily a linear matroid in the sense of Example \ref{ex:linear-matroids}. In this section we study some properties and invariants of $\M_n$.

\begin{definition}
    The \emph{chromatic polynomial} of a matroid $(E,I)$ is given by
    \[ p_E(t) = \sum_{S\subseteq E} (-1)^{|S|} t^{r(E)-r(S)} \]
    where $r$ is the rank function, i.e.\ size of smallest independent subset.
\end{definition}

\begin{definition}
    The \emph{Tutte polynomial} of a matroid $(E,I)$ is given by
    \[ T_E(x, y) = \sum_{S\subseteq E} (x-1)^{r(E)-r(S)}(y-1)^{|S|-r(S)}. \]
\end{definition}

The Tutte polynomial contains strictly more information than the chromatic polynomial, since they are related by
\[ p_E(t) = (-1)^{r(E)} T_E(1-t, 0). \]

\begin{theorem}
The Tutte polynomial of $\M_n$ is
\[ T_{\M_n}(x, y) = \prod_{k=1}^n (x-1+(k)_y) \]
where $(k)_y \defeq \frac{y^k-1}{y-1}$.
\end{theorem}

\begin{proof}
  The base case $n=0$ is trivial. For the induction step, we write $\{0,\dotsc,n+1\}=\{0\}\cup\{1,\dotsc,n+1\}$; we invite the reader to see what goes wrong if one tries to induct via $\{0,\dotsc,n\}\cup\{n+1\}$.

  Let $S$ be a subset of $[n+1]^\to$, and let $S'$ denote its restriction to $([n]')^\to$, where $[n]'=\{1\le\dotsc\le n+1\}$. Write $A=[n+1]^\to - ([n]')^\to$ for the set of edges in $[n+1]$ with source $0$. We distinguish two cases:
  \begin{itemize}
      \item If $S=S'$, then $r(S)=r(S')$ and $|S|=|S'|$. 
      
      \item If $S=S'\cup F$ for some $\emptyset\ne F\subset A$, then $r(S)=r(S')+1$ and $|S|=|S'|+|F|$. 
  \end{itemize}
  Consequently, we have
  \begin{align*}
    \frac{T_{\M_{n+1}}(x, y)}{T_{\M_n}(x,y)}
      &= (x - 1) + \sum_{\emptyset\ne F\subset A} (y-1)^{|F|-1}\\
      &= \fade{(x - 1) +{}} \frac{\left(\sum_{F\subset A}(y-1)^{|F|}\right)-1}{y-1}\\
      &= \fade{(x - 1) +{}} \frac{((y-1)+1)^{|A|}-1}{y-1}\\
      &= \fade{(x - 1) +{}} (n+1)_y \qedhere
  \end{align*}
\end{proof}

\begin{corollary}
  The chromatic polynomial of $\M_n$ is
  \[ p_{\M_n}(t) = (t-1)^n. \]
\end{corollary}

Finally, we compute the dual of $\M_n$. Recall that the dual $M^*$ of a matroid $M$ has the same elements of $M$, and the basis sets of $M^*$ are defined to be the complements of the basis sets of $M$.

\begin{theorem}
  Let $X_{i\to}=\{e \in X \mid \Src(e) = i\}$.
  A subset $X\subset\M_n^*$ is independent if and only if $ |X_{i\to}| < n-i $ for each $0\le i<n$. It is a basis set if and only if $|X_{i\to}| = n-i-1$ for each $0\le i<n$.
\end{theorem}
\begin{proof}
  By Remark \ref{rmk:rank}, the basis sets of $\M_n$ are those subsets $X\subset [n]^\to$ containing exactly one edge of source $i$ for each $0\le i<n$. Therefore, a basis set of $\M_n^*$ must contain all but one of the $n-i$ edges emanating from $i$.
\end{proof}

\begin{remark}
  In \cite{TransferSystemsDuality}, it is shown that the lattice of transfer systems $\Tr(G)$ is self-dual whenever $G$ is abelian. For example, the duality for $\Tr(C_{p^2})$ is given as follows.
  
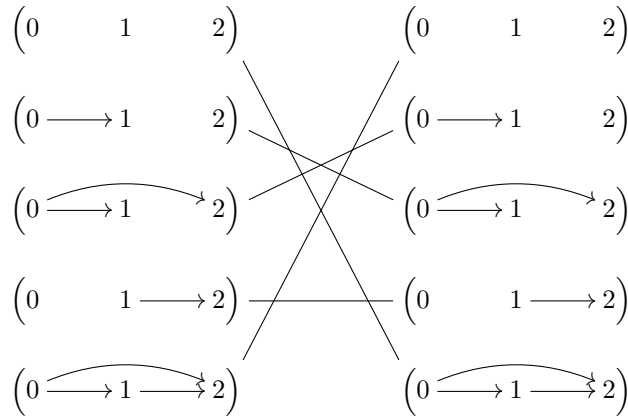
\[\xymatrix@R=1em{
      \Big(0 & 1 & 2\Big) \ar@{-}[ddddrr] && \Big(0 & 1 & 2\Big)\\
      \Big(0 \ar[r] & 1 & 2\Big) \ar@{-}[drr] && \Big(0 \ar[r] & 1 & 2\Big)\\
      \Big(0 \ar[r] \ar@/^1em/[rr] & 1 & 2\Big) \ar@{-}[urr] && \Big(0 \ar[r] \ar@/^1em/[rr] & 1 & 2\Big)\\
      \Big(0 & 1 \ar[r] & 2\Big) \ar@{-}[rr] && \Big(0 & 1 \ar[r] & 2\Big)\\
      \Big(0 \ar[r] \ar@/^1em/[rr] & 1 \ar[r] & 2 \Big) \ar@{-}[uuuurr] && \Big(0 \ar[r] \ar@/^1em/[rr] & 1 \ar[r] & 2 \Big)
  }\]
  We do not see a clear relation between this notion of duality and the matroid duality.
\end{remark}

%
%
%
%


\newcommand{\etalchar}[1]{$^{#1}$}
\providecommand{\bysame}{\leavevmode\hbox to3em{\hrulefill}\thinspace}
\providecommand{\MR}{\relax\ifhmode\unskip\space\fi MR }
\providecommand{\MRhref}[2]{%
  \href{http://www.ams.org/mathscinet-getitem?mr=#1}{#2}
}
\providecommand{\href}[2]{#2}

\end{document}